\documentclass[11pt]{amsart}
\usepackage{graphicx}
\usepackage{graphics}

\usepackage{color}

\newcommand{\szego}{Szeg\H{o}\;}

\newcommand{\kahler}{K\"ahler }

\newcommand{\C}{{\mathbb C}}

\newcommand{\dbar}{\bar\partial}
\newcommand{\ddbar}{\partial\dbar}

\newcommand{\D}{{\mathbf D}}

\newcommand{\half}{{\frac{1}{2}}}

\newcommand{\hcal}{\mathcal{H}}

\newcommand{\lcal}{\mathcal{L}}

\newcommand{\al}{\alpha}

\newcommand{\la}{\lambda}

\newcommand{\pa}{\partial}

\newtheorem{maintheo}{{\sc Theorem}}

\newtheorem{theo}{{\sc Theorem}}[section]

\newtheorem{lem}[maintheo]{{\sc Lemma}}

\newenvironment{defin-no-number}{\medskip\noindent{\it Definition:\/} }{\medskip}

\title[Szeg\H o kernels and Poincar\'e series]{Szeg\H {o}  kernels and Poincar\'e series}

 \author{Zhiqin Lu} \address{Department of
Mathematics, University of California,
Irvine, Irvine, CA 92697, USA} \email[Zhiqin Lu]{zlu@uci.edu}

\author{Steve Zelditch}
\address{Department of Mathematics, Northwestern  University,
Evanston, IL 60208-2370, USA} \email{
zelditch@math.northwestern.edu}

\thanks{The first author is partially supported by  NSF grant \# DMS-12-06748 and the second author is partially supported by NSF grant  \#  DMS-1206527.}


\date{August 12, 2013}

\begin{document}

\begin{abstract} Let $M = \tilde{M}/\Gamma$ be a \kahler manifold, where $\Gamma \simeq \pi_1(M)$
and where $\tilde{M}$ is the universal \kahler cover. Let $(L, h) \to M$ be a positive Hermitian 
holomorophic line bundle. We first prove that the $L^2$ \szego
projector $\tilde{\Pi}_N$ for $L^2$-holomorphic sections on the lifted bundle $\tilde{L} ^N\to \tilde{M}$ is 
related to the \szego projector for $H^0(M, L^N)$ by  $\hat{\Pi}_N(x, y) = \sum_{\gamma \in \Gamma} \tilde{\hat{\Pi}}_N(\gamma \cdot
x, y). $ We apply this result to give a simple proof of Napier's theorem on the holomorphic convexity
of $\tilde{M}$ with respect to $\tilde{L}^N$ and to surjectivity of Poincar\'e series.
\end{abstract}

\maketitle

 Let $(M, \omega)$ denote a compact \kahler manifold of dimension $m$, and let
$(\tilde{M}, \tilde{\omega})$ denote its universal \kahler cover
with deck transformation group $\Gamma=\pi_1(M)$. 
We assume that $\Gamma$ is an infinite group so that $\tilde M$ is complete noncompact. 
Let $(L, h) \to
(M,\omega)$ denote a positive hermitian  line bundle  and let
$(\tilde{L}, \tilde{h})$ be  the induced hermitian line  bundle
over $\tilde{M}$.  The  first purpose of this note is to prove that  for sufficiently large $N \geq N_0(M, L, h)$, the 
 \szego kernel\footnote{In the context of positive
line bundles, the  \szego kernel and Bergman kernel are essentially the same and we use the two terms
interchangeably.}   of the holomorphic projection
$\Pi_{h^N}: L^2(M, L^N) \to H^0(M, L^N)$ on the quotient is given by the Poincar\'e series
of the  \szego projection 
for $L^2$ holomorphic sections on the universal cover (Theorem \ref{UPSZEGON}). This relation is standard in the theory of
the Selberg trace formula on locally symmetric spaces, but seems not to have been proved before in the
general setting of positive line bundles over \kahler manifolds. As will be seen, it is a consequence
of  standard Agmon estimates on off-diagonal decay of the \szego kernel  \cite{Del,L} and  of the local structure
of the kernel given by the  Boutet de Monvel-Sj\"ostrand parametrix for both Szeg\H o kernels \cite{BouSj, BBSj}.
This relation is then used   to simplify and unify a number of
results on universal covers of compact \kahler manifold.  One application is  a short proof of
the holomorphic convexity with respect to the positive line bundle $(\tilde{L}, \tilde{h})$ (Theorem \ref{N})  proved by  T. Napier  \cite{N}.
A second application is a simple  proof of surjectivity of Poincar\'e series (Theorem \ref{POINCARE}).  The  problem of determining the least $N_0(M, L, h)$ for which these results are true is not
treated in this article.

To state the results, we need to introduce some notations.
For any positive hermitian line bundle $(L, h) \to (M,\omega)$ over a
\kahler manifold, we denote by $H^0(M, L^N)$ the space of
holomorphic sections of the $N$-th power of $L$.  We assume throughout that
$\omega : = -\frac{i}{\pi} \ddbar \log h$ is a \kahler metric. The Hermitian metric $h$ induces  the
inner products
\begin{equation} \langle s_1, s_2 \rangle_{h^N} = \int_{M}
(s_1(z), \overline{s_2(z)})_{h^N} dV_{M}, \end{equation} where the volume form\footnote{We sometimes also use $dV_M(y)$ instead of $dV_M$ if we want to specific a variable $y$.}
is given by $dV_{M} = \omega^m/m!$. The corresponding inner product for $\tilde h$ can be defined in the similar way. We also use 
$|\cdot|_{h^N}$ and $|\cdot|_{\tilde h^N}$ to denote the pointwise norms of the metrics $h$ and $\tilde h$, respectively. 
We use the notation
 $ H^0_{L^2}(\tilde{M}, \tilde{L}^N)$ for the space of
$L^2$ holomorphic sections.  More generally, we  denote
the space of $L^p$ holomorphic sections by  $ H^0_{L^p}(\tilde{M}, \tilde{L}^N)$ for $p\geq 1$.
We further  denote by
\begin{equation} \label{PIN}  \Pi_{h^N} : L^2(M, L^N) \to H^0(M, L^N)
\end{equation}
the orthogonal projection or the \szego kernel with respect to
$\langle\,\,,\, \,\rangle_{h^N}$ and by
\begin{equation} \tilde{\Pi}_{h^N} : L^2(\tilde{M}, \tilde{L}^N) \to
 H^0_{L^2}(\tilde{M}, \tilde{L}^N)
\end{equation}
the corresponding  orthogonal projection or the   \szego kernel on the universal
cover $\tilde M$. Since $h$ is fixed in this discussion we often simplify the notation by
writing $\Pi_N$ and $\tilde{\Pi}_N$.  

It has been proved by  Delin  \cite{Del},
Lindholm \cite{L} and Christ \cite{Ch} 
in various settings that the \szego kernels admit the following 
{\em Agmon} estimate: there exists $\beta= \beta(\tilde M, \tilde L, \tilde h)>0$ such that
\begin{equation}\label{AGMON}  |\tilde{\Pi}_N(x, y)|_{\tilde{h}^N}  \leq e^{- \beta
\sqrt{N} d(x, y)} \end{equation}
for $d(x,y)\geq 1$, where $d(x, y)$ denotes the distance function of $\tilde M$.  We use this relation
to study the Poincar\'e series map

\begin{equation} \label{P} P : H^0_{L^1}(\tilde{M}, \tilde{L^N}) \to  H^0(M, L^N), \;\;\;\;\;  P
f(z) := \sum_{\gamma \in \Gamma} f(\gamma\cdot z)  \end{equation}
for any  $N>0$.


Each deck transformation $\gamma \in \Gamma$ determines a displacement function
$d_{\gamma}(x) = d(x, \gamma x)$ on $\tilde{M}$. Its minimum value  is denoted $L(\gamma)$. The minimum set
is the axis of $\gamma$. When it is of dimension one,  it folds up under $\Gamma_{\gamma}$, the
centralizer of $\gamma$ in $\Gamma$, to a closed geodesic and $L(\gamma)$ is its length; in degenerate cases, it
 equals the common  length of the closed geodesics (cf. \cite[page 95]{ec}).  

 We denote by $\ell_1$ the minimum over $\gamma \in \Gamma$ of $L(\gamma)$, i.e. 
the length of the shortest closed geodesic of $(M, \omega)$. 

The main result of this paper is the following $\sum_\gamma$ relation.

\begin{maintheo} \label{UPSZEGON} 
There is an integer  $N_0=N_0(M,L,h)$, such that if $N\geq N_0$, then 
the degree $N$ \szego kernel $\Pi_N(x,y)$ of $(L, h) \to M$ and  $\tilde\Pi_N(x,y)$ of $(\tilde{L} ,\tilde{h})\to \tilde{M}$ are related by
$${\Pi}_N(x, y) = \sum_{\gamma \in \Gamma} \tilde{{\Pi}}_N(\gamma \cdot
x, y). $$
\end{maintheo}

 There is a classical proof of Theorem \ref{UPSZEGON}
(due to Selberg, Godement and Earle \cite{Ear})
for  a bounded symmetric domain,  which proves the result under  the additional assumption on the variation of the Bergman kernel. We also note \footnote{More generally, the relation of solutions of elliptic equations on both $M$  and $\tilde M$ (cf. ~\cite{A}).}  that the analogue of the  $\sum_{\gamma}$ relation   
of  Theorem \ref{UPSZEGON}   for heat kernel, or for the wave kernel $\cos t \sqrt{\Delta}$, is  simpler 
and standard.  Consider the heat kernel  $\tilde{{\hat{K}}}(t, x, y)$ for the heat operator $\exp (- t \Box_b)$ generated by the Kohn Laplacian $\Box_b$  associated to $(L, h)$. For
each $N$, let $\tilde{{\hat{K}}}_N(t, x, y)$ be the component of $\tilde{{\hat{K}}}(t, x, y)$ identified with  the Kodaira Laplacian of the line bundle $L^N$.  We easily see that
\begin{equation} \label{HEAT} {\hat K}_N(t, x, y) = \sum_{\gamma \in \Gamma} \tilde{{\hat{K}}}_N(t, \gamma \cdot
x, y).  \end{equation}  To prove \eqref{HEAT}, we just  note that both sides solve the heat  equation and that
they have the same initial condition, i.e. the  delta function.  In
the case of Szeg\"o kernels, it is also simple to see that both sides are holomorphic projectors. The same
argument can be used for the wave kernel, whenever the $\sum_{\gamma}$ is finite. But it is not
apriori clear that  the right side is a surjective projection onto $H^0(M, L^N)$. Surjectivity is a kind of replacement
for the initial condition in the case of the heat kernel, but  it is a more complicated kind of \emph{boundary condition}. 
As mentioned above, Theorem \ref{UPSZEGON} is proved by studying the singularity on the diagonal of the Szeg\H o kernel
using  the parametrix construction in \cite{BouSj,BBSj}.  The main idea is that the $\sum_{\gamma}$ relation must
hold because the $L^2$ Szeg\H o kernel has precisely the same local singularity as the quotient Szeg\H o kernel. 
In principle it is possible to estimate $\beta$ and $\ell_1$ and hence
to estimate the minimal power for which the relation is valid.

An alternative approach to  Theorem \ref{UPSZEGON}, which we do not carry out here,  is to use \eqref{HEAT}
and take  the limit $t \to \infty$. One  needs to take
the limit $t \to \infty$ under the summation sign $\sum_{\gamma}$ on the right side and show that one obtains
the Szeg\H o kernel part of each term. It seems that the limit $t \to \infty$ is monotone decreasing along the
diagonal $x = y$ so that the limit
may indeed  be taken under the $\sum_{\gamma}$ side.  One may then  use the existence of a spectral gap for the Kohn Laplacian both upstairs and downstairs  to
show that the limit of the left side of \eqref{HEAT} is the downstairs Szeg\"o kernel and the limit of each term
on the right side is the corresponding expression for the upstairs Szeg\"o kernel. However, we opted to work
entirely with the Szeg\"o kernel.

Our first application is to give a simple proof of the holomorphic
convexity of $\tilde{M}$ with respect to sufficiently high powers
of a positive line bundle. We recall that $\tilde{M}$ is called
holomorphically convex if for each sequence $\{x_n\}$ with no
convergent subsequence, there exists a holomorphic function $f$ on
$\tilde{M}$ such that $|f(x_n)|$ is unbounded. It is
holomorphically convex with respect to $\tilde{L}^N$ if there
exists $s \in H^0(\tilde{M}, \tilde{L}^N)$ such that
$|s(x_n)|_{{\tilde h^N}}$ is unbounded.

In \cite{N}, T. Napier proved a special case of the Shafarevich
conjecture, which states the holomorphic convexity of the universal cover of certain complex manifolds, and he also proved holomorphic convexity with respect to
high powers of a positive line bundle.
The recent development of the conjecture can be found in~\cite{EKPR}.
Note that holomorphic convexity  is much simpler to prove with the presence of high powers of a positive line bundle.
 In \S \ref{HC}, we give a
new proof of the following theorem:

 \begin{maintheo}[\cite{N}] \label{N}  Let $M$ be a connected smooth projective
variety and $L \to M$ a positive holomorphic line bundle. Then
there exists an integer $N_0=N_0(M,L,h)$ so that for $N \geq N_0$, the
universal cover $\tilde{M}$ is holomorphically convex with respect
to  $\tilde{L}^N$. \end{maintheo}

Our second application is the  surjectivity of
 the Poincar\'e map \eqref{P}.  In general, the operator $P$ is not surjective. We prove

\begin{maintheo} \label{POINCARE} Suppose that $N$ is large enough so that Theorem~\ref{UPSZEGON}  holds. Then the Poincar\'e
map is surjective from $H^0_{L^1}(\tilde{M}, \tilde{L}^N) \to H^0(M, L^N). $ \end{maintheo}

As discussed above, surjectivity is the non-trivial aspect of the $\sum_{\gamma}$ relation,  and as we show
in \S \ref{SURJ} it is an almost immediate consequence of  Theorem \ref{UPSZEGON}. The original motivation for this article
was to simplify the discussion of surjectivity in Kollar's book \cite{K}.  We briefly review its approach.

 In
\cite[Theorem 7.12]{K},  $P$ is proved to be surjective as long as
the Bergman kernel on $L^2$ extends to $L^1$ and $L^{\infty}$ and
is a reproducing kernel on $L^{\infty}$.   Koll\'ar reviews
two conditions (7.9) (Condition 1) and (7.11) (Condition 2) under which surjectivity
was proved by Earle \cite{Ear}.
Condition 1 is that the Bergman projection $\tilde{\Pi}_N$ for $(\tilde{M}, \tilde{h})$ extends
to bounded linear maps on $L^1(\tilde{M}, \tilde{L}^N)$ and  $L^{\infty} (\tilde{M}, \tilde{L}^N)$. 
As verified in \cite[Proposition 7.13]{K}, it is sufficient that $\tilde{\Pi}_N(\cdot, w) \in L^1(\tilde{M})$
with $\|\tilde{\Pi}_N(\cdot, w)\|_{L^1} \leq C$ for a uniform constant $C$ independent of $w$. Condition 
2 is that $\tilde{\Pi}_N$ is a reproducing kernel on all $L^{\infty}(\tilde{M}, \tilde{L}). $
The  Agmon estimates are sufficient to ensure Condition 1.   In \cite[Proposition 7.14]{K}, a rather
strange condition is used to prove Condition 2: namely that $\frac{\Pi_{2
N}(z,w)}{\Pi_N(z,w)}$ is in $L^2$. In particular, that
$\Pi_N(z,w)$ is never zero.
 Koll\'ar writes {\em the conditions in (7.14) are ....quite artificial; it
is not clear...how restrictive condition 2 is in reality}. Theorem \ref{UPSZEGON}  shows that the natural restriction
is  convergence of the Poincar\'e series in its statement.  We do not study in detail  the question of effective estimates,
i.e. of finding the minimal value of $N_0$, but the proof of Theorem \ref{UPSZEGON}  shows that $N_0$ is determined
by balancing the growth rate of $\Gamma$ with optimal off-diagonal estimates \eqref{AGMON} on $\Pi_N(x,y)$.
 In ~\cite{Y1,Y2,Y3}, Yeung proved some effectiveness results for towers of  Galois covers over a  \kahler  manifold.
His techniques may prove to be   useful in obtaining an effective estimate of $N_0$.  We hope to study this question in a future  article.

There exist many additional articles devoted to  universal covers of K\"ahler manifolds and the relations between
the complex geometry above and below. See for instance \cite{G1,G2,E,Kai,Ca,donn,Y1,Y2}.    But to our knowledge,  
they do not use the relation of  Theorem
\ref{UPSZEGON}.

\section{Bergman/\szego kernels}\label{s1}

In this section, we review the definition of the Bergman/\szego kernel for a positive Hermitian holomorphic
line bundle $(L,h) \to M$. We also go over a basic example where an explicit formula on the universal cover exists.

The \szego kernel of $(L, h) \to M$ is the Schwartz kernel of the orthogonal projection \eqref{PIN}. To
obtain a Schwartz kernel we need to introduce a local holomorphic frame $e_L$ over an open set $U \subset M$.
Then a local holomorphic section may be written $s = f e_L$ where $f $ is a local holomorphic function on $U$.
Similary, $e_L^{\otimes N}$ is a local holomorphic frame for $L^N$. 
We choose an orthonormal basis $\{S^N_j\}$  of $H^0(M, L^N)$ and write
$S^N_j=f_j e_L^{\otimes N}:j=1,\dots,d_N$ where $d_N = \dim H^0(M, L^N)$. 
Then the \szego kernel $\Pi_{N}(z,w)$ for $(L^N, h^N)$ relative to $dV_M$
is the section of $L^N \otimes \overline{L}^N \to M \times M$  given  by   \begin{equation}\label{FNdef}  \Pi_{N}(z,\bar w): = B_N
(z,\bar w)\,e_L^{\otimes N}(z) \otimes\overline {e_L^{\otimes
N}(w)}\,,\end{equation} where
\begin{equation}\label{FN}B_N(z,\bar w)=
\sum_{j=1}^{d_N}f_j(z) \overline{f_j(w)}\;.\end{equation}
In \cite{BBSj}, $B_N(z,w)$ is called the Bergman kernel.

 \subsection{ Szeg\H o kernel for line bundles and the associated disc bundle}

Instead of using local frames, one can define scalar kernels if one lifts
the \szego kernels  to the unit frame bundle
$X$ associated to the dual Hermitian line bundle $(L^*, h^*) \to M$ of $(L, h)$. The behavior of the lifts under translations by $\gamma \in \Gamma$ are somewhat  more   transparent than
for $\Pi_{h^N}(z,w)$ which is a section of $L^N \otimes \bar{L}^N$. The choice whether to use the 
kernels on $M \times M$ or their lifts to $X \times X$ is mainly a matter of convenience.  
 In this section, we review the lift
to the unit frame bundle on $M$ and $\tilde{M}$.

As above,
$L^*$ denotes the dual line bundle to $L$. The hermitian metric $h$ on $L$ induces the dual metric
$h^*(=h^{-1})$ on $L^*$.  We define the principal
$S^1$  bundle $\pi: X \to M$ by $$X=\{\la \in L^* : \|\la\|_{h^*}= 1\} = \partial D, \;\; \mbox{where}\;\; D  = \{\la \in L^* : \rho(\la)>0\}, $$
where
$\rho(\la)=1-\|\la\|^2_{h^*}$. We let $r_{\theta}x =e^{i\theta} x$
($x\in X$) denote the $S^1$ action on $X$ and denote its
infinitesimal generator by $\frac{\partial}{\partial \theta}$. The disc bundle $D$ is strictly
pseudoconvex in $L^*$, since the curvature $\Theta_h$ of $h$  is positive, and hence $X$
inherits the structure of a strictly pseudoconvex CR manifold.
Associated to $X$ is the contact form $\al=
-i\partial\rho|_X=i\dbar\rho|_X$ and  the volume
form
\begin{equation}\label{dvx}dV_X=\frac{1}{m!}\al\wedge
(d\al)^m=\pi^m\,\al\wedge\pi^*dV_M\,.\end{equation}
It induces the $\lcal^2(X, dV_X)$  inner product
\begin{equation}\label{unitary} \langle  F_1, F_2\rangle
=\frac{1}{2\pi^{m+1}}\int_X F_1\overline{F_2}dV_X.\end{equation}

A section  $s_N$ of $L^N$ lifts  to an
equivariant function $\hat{s}_N$ on $L^*$, defined by
$$\hat{s}_N(\lambda) = \left( \lambda^{\otimes N}, s_N(z)
\right)\,,\quad \la\in L^*_z\,,\ z\in M$$ 
We
henceforth restrict $\hat{s}_N$ to $X$ and then the equivariance
property takes the form $\hat s_N(r_\theta x) = e^{iN\theta} \hat
s_N(x)$.  We may express the lift in local coordinates $z$ on $U \subset M$
and in a local holomorphic frame $e_L: U \to M$. They induce
local coordinates $(z,\theta)$ on $X$ by the rule
$x=e^{i\theta}|e_L(z)|_he_L^*(z)\in X$. The equivariant lift of
a section $s=fe_L^{\otimes N}\in H^0(M,L^N)$ is then given
by
\begin{equation}\label{lifta}\hat s(z,\theta) =
e^{iN\theta} |e_L|_h^N f(z) = e^{N\left[-\half
\phi(z) +i\theta \right]} f(z)\;,\end{equation}
where $|e_L(z)|_h = e^{- \frac 12\phi(z)}$ and $\phi(z)$  is the local \kahler potential.
The map $s\mapsto \hat{s}$ is a unitary equivalence
between $L^2(M, L^{ N})$ and $\lcal^2_N(X, dV_X)$, where $\lcal^2_N(X, dV_X) \subset \lcal^2(X, dV_X)$
is the subspace of equivariant functions transforming by $e^{i N \theta}$ under $r_{\theta}$.

The   Hardy
space $\hcal^2(X) \subset \lcal^2(X, dV_X)$ is by definition
the subspace of functions that are annihilated by the
Cauchy-Riemann operator $\dbar_b$. The
$S^1$ action on $X$ commutes with $\bar{\partial}_b$ and hence  the subspace $\hcal^2_N(X) \subset \hcal^2(X)$
of equivariant CR functions is the intersection $\hcal^2(X) \cap \lcal^2_N(X, dV_X)$.
 The lift of $s_N \in H^0(M, L^N)$
is then an  equivariant  CR function $\hat{s}_N \in \hcal^2(X)$,
    hence
$\hcal^2(X) = \bigoplus_{N =0}^{\infty} \hcal^2_N(X)$. The Szeg\H o kernel $\hat{\Pi}$ is the (distribution) kernel of
the orthogonal projection $L^2(X) \to \hcal^2(X)$. 

The  \szego kernels  $\Pi_{N}$ lift to  equivariant scalar kernels
$\hat{\Pi}_{N}$ on $X \times X$, with $\hat{\Pi}_{N}$  the Schwartz kernel of the  orthogonal
projection $\hat{\Pi}_{N} : \lcal_N^2(X, dV_X)\rightarrow \hcal^2_N(X)$,
defined by
\begin{equation} \hat{\Pi}_{N} F(x) = \int_X \hat{\Pi}_{N} (x,y) F(y) dV_X (y)\,,
\quad F\in\lcal^2(X, dV_X)\,. \label{PiNF}\end{equation}  Then $\hat{\Pi}_{N} $ is the $N$th Fourier
coefficient of $\hat{\Pi}$ and in terms of the orthogonal decomposition above, we have  
\begin{equation} \label{SUMN} \hat{\Pi} = \sum_N \hat{\Pi}_N  \end{equation}
as operators on $L^2(X)$.

Using~\eqref{FNdef}, the Bergman kernel $\hat\Pi_N$
can be given as
\begin{equation}\label{szego2} \hat{\Pi}_{N}(x,y) =\sum_{j=1}^{d_N}
\hat{S}_j^N(x)\overline{ \hat{S}_j^N(y)}\,,\end{equation} where
$S_1^N,\dots,S_{d_N}^N$ form an orthonormal basis of $H^0(M,L^N)$.
By~\eqref{lifta}, the  lifted \szego
kernel is given in terms of the Bergman kernel on $U \times U$ by
\begin{equation}\label{szegolift}\hat{\Pi}_{N}(z,\theta;w,\phi) =
e^{N\left[-\half \phi (z)-\half \phi(w)
+i(\theta-\phi)\right]} B_N(z,\bar w)\;.\end{equation}

Theorem \ref{UPSZEGON} can be restated as follows in terms of the 
 \szego kernels on the unit circle bundle:
\begin{equation} \label{UPCB} \hat {\Pi}_N(x, y) = \sum_{\gamma \in \Gamma} \tilde{{\hat \Pi}}_N(\gamma \cdot
x, y).  \end{equation}
In this formulation, translation by $\gamma$ acts on a scalar kernel rather than a section of a line bundle. By
\eqref{SUMN} one has a similar Poincar\'e series formula for $\hat{\Pi}_N$.

\subsection{\szego kernels for the hyperbolic disc}

To illustrate the notions above, we consider the  familiar example of the  lifted \szego kernels  on the hyperbolic disc $\D$.
In this case, the positive line bundle is the canonicl bundle  $L = T^{*(1,0)}\D$ equipped
 with the hyperbolic
hermitian metric $h_{\D}$ dual to the hyperbolic metric on $T^{(1,0)} \D$.
There exists a    global
holomorphic frame $dz$ for $L \to \D$ with Hermitian norm
$\|dz\|_{h_{\D}}^2 = (1 - |z|^2)^2.$ Hence the \kahler potential is given
by $\phi(z) = \log (1 - |z|^2)^{-2}.$
Thus for $s_N = f (dz)^N$, one has
$\| f (dz)^N \|_{h_{\D}}^2 = |f(z)|^2 (1 - |z|^2)^{2N}. $
The dual bundle $L^*$ is  $T^{(1,0)} \D$ with  the
usual hyperbolic metric, so that
$X = \{(z, v) \in T^{(1,0)}\D: |v|_z = 1\}$
is  the unit tangent bundle of $\D$, i.e. equals $P SU(1,1)$.
 In the local coordinates  $(z,\theta)$ on $X$ denote the coordinates of the point
$x=e^{i\theta} |\frac{\partial}{\partial z}|^{-1}_h
\frac{\partial}{\partial z} \in  X$, we have
\begin{equation}\label{lift}\hat s(z,\theta) =
e^{iN\theta} |dz|^N_h f(z) = e^{
iN \theta } (1 - |z|^2)^{N} f(z)\;.\end{equation}

The Bergman kernel for $L^t$  (denoted in \cite{Ear} by $k_t$)  is explicitly given by
$$2\,k_t(z, w) = (1 - z \bar{w})^{- 2 t} =
\sum_{j = 0}^{\infty} (2t)_j  \frac{(z \bar{w})^j}{j!}, \;\;$$
where
$t_j = (t-1)\,t\, (t + 1) \cdots (t + j - 1). $ The fact that $k_t = k_1^t$ is a reflection of the homogeneity
of $\D$. Furthermore,  $k_1(z,z) dz  = dm = dV_h$  and
$k(z,z)^{-1/2} = e^{- \phi/2} = (1 - |z|^2)$ when $L = T^{* (1,0)}\D$.
In the notation above, $k_t(z,w) = F_{h_{\D}^t}(z,w),$ where $F_t$ is the local Bergman kernel in the frame
$(dz)^t$. The lifted Bergman kernel is give by
\begin{equation}\label{szegoD} \hat{\Pi}_{h_{\D}^t} (z, 0; w, 0)=
 C_{m} \left(\frac{(1 - z \cdot \bar{w})}{\sqrt{1 - |z|^2}
\sqrt{1 - |w|^2}} \right)^{- 2t} \end{equation}
for constant $C_m$ depending only on $m$.

We also refer to \cite{Ear} for calculations in the  general  setting of a bounded homogeneous domain $B$ with $L = K$ (the canonical bundle).  The Bergman kernel in that setting is induced by 
the natural inner product on $H^0(B, K)$, i.e. on $(n, 0)$ forms and the Bergman kernel
  $k(z,\zeta) dz \otimes {d\bar\zeta}$ is naturally an $(n, n)$ form.

\subsection{Bergman/\szego kernel on  $\tilde{M}$ and the action of $\Gamma$}

We now consider the \kahler cover $\pi: (\tilde{M}, \tilde{\omega}) \to (M, \omega)$. By definition,
$\tilde{L} = \pi^* L$ and $\tilde{h} = \pi^*h$. We then define the unit circle bundle $\tilde{X} \to \tilde{M}$ similarly.

Because $\tilde{M}$ is simply connected,  $\Gamma$ automatically lifts to $\tilde{X}$ as
a group of CR holomorphic contact transformations with respect to $\alpha$, and in particular the action
of $\Gamma$ linearizes on the spaces $ H_{L^2}^0(\tilde{M}, \tilde{L}^N)$. We briefly recall the proof: by
assumption, $\gamma \in \Gamma$ is an isometry of $\tilde{\omega}$ and thus is a symplectic transformation.
We claim that
 $\gamma$ preserves the holonomy map of the connection
$1$-form $\alpha$, i.e. the map
$H(\beta) = e^{i \theta_{\beta}} $
   defined by horizontally lifting a loop
$\beta: [0, 1] \to M$ to $\tilde{\beta}: [0, 1] \to X$ with respect to $\alpha$  and
expressing $\tilde{\beta}(1) = e^{i \theta_{\beta}}
\tilde{\beta}(0). $
Then  $\gamma$ preserves the  holonomy-preserving in the sense tht
$ H(\gamma (\beta)) = H(\beta)$ for all loops $\beta$.  Indeed, we may assume that  the loop is contained in the domain
of a local frame $s: U \to X$, and  then
$ H(\beta) = \exp (2 \pi i \int_{\beta} s^* \alpha ).
$
But $\tilde{M}$ is simply connected so that  $\beta = \partial \sigma$ and  $\int_{\beta} s^* \alpha =
\int_{\sigma} \omega$. Since $\gamma$ is   symplectic, it thus preserves
the holonomy around homologically trivial
loops and all loops on $\tilde{M}$ are trivial.

Since $\Gamma$ acts by holomorphic transformations of $\tilde{M}$, it lifts to a group of $CR$ 
maps on $\tilde{X}$ which commmute with the $S^1$ action.
It is easy to see that $\tilde{\Pi}_{N}$ commutes with the action of $\Gamma$ on $\hcal^2_N(X)$, hence
\begin{equation}\label{invszego}
\tilde{\Pi}_{N}(\gamma x, \gamma y) = \tilde{\Pi}_{N}(x, y). 
\end{equation}

This identity is often written as a transformation law for the scalar  \szego kernel $\hat\Pi_N$ of a local frame under $\Gamma$.
In  most  works such as \cite{Ear}, $\tilde{M}$ is contractible and $\tilde{L} \to \tilde{M}$ is holomorphically
trivial, hence there exists a global frame  $\tilde{e}_L$. Since $\Gamma$ linearizes on $H^0(\tilde{M}, \tilde{L}^N)$,
there exists a function $J(\gamma, z)$ (a factor of automorphy) such that $\gamma^* \tilde{e}_L = J(\gamma, z) \tilde{e}_L$.
Then $$B_N(\gamma z, \gamma w) = J(\gamma, z) \overline{J(\gamma, w)} B_N(z, w). $$

\section{Proof of Theorem \ref{UPSZEGON}}

In this section we prove Theorem \ref{UPSZEGON}.

\subsection{Agmon estimates}

We first sketch the proof of the following Agmon estimate for the
\szego kernel, which is almost entirely contained in the previous
work of Delin, Lindholm and Berndtsson.

\begin{theo}[\cite{Del,L,BBSj} (See Theorem 2 of \cite{Del} and Proposition 9 of \cite{L})] \label{AGMON1}  Let $M$ be a compact
\kahler manifold, and let $(L, h) \to M$ be a positive Hermitian line
bundle.  Then the exists a constant $\beta=\beta(M,L,h)>0$ such that
$$|\tilde\Pi_N(x, y)|_{\tilde h^N} \leq e^{-\beta \sqrt{N} {d} (x, y)}, \;\; \mbox{for}\;\; {d} (x, y) \geq 1.  $$
where ${d}(x,y)$ is the Riemannian distance with respect to
the \kahler metric $\tilde{\omega}$.\end{theo} 

\begin{proof}[Review of the proof] In ~\cite[Proposition 9]{L}, the following is proved for a strictly pseudo-convex domain of $\C^m$. The same argument works on strictly pseudo-convex manifold. In our notation, it was proved that 
$$ |\tilde\Pi_N(x,y)|^2_{\tilde h^N} \leq CN^{2m}e^{- \epsilon\sqrt N d(x,y)} $$
for some $\epsilon>0$.
Since $d(x,y)\geq 1$, the polynomial term of $N$ can be absorbed by the exponential term by shrinking $\epsilon$.\end{proof}

\subsection{$\dbar$ estimates and existence theorems on complete \kahler manifolds} The following H\"ormander's $\bar\pa$ estimate is essential in our proof of Theroem~\ref{UPSZEGON}. 

\begin{theo}\label{demai} Let $(X, \omega)$ be a complete \kahler manifold, and let $L \to X$ be a hermitian line
bundle with the hermitian metric $h$. Assume that there is an integer $N_0$ such that   the curvature $\sqrt{-1}N_0\,  \Theta(h)+\textup{Ric}(\omega)\geq c\,\omega$ is positive for some $c>0$. Then for any $N\geq N_0$,  the following is true: for any $g \in L^2(X, \bigwedge^{0,1}\otimes L^N) $ satisfying $\dbar g = 0$,
and $\int_X |g|_{h^N}^2{\omega^n} < \infty$, there exists $f \in L^2(X, L^N) $ such that $\dbar f = g$ and
$$\int_X |f|_{h^N}^2 \omega^n \leq c^{-1}\int_X  |g|_{h^N}^2 \omega^n. $$
\end{theo}

\subsection{Bergman kernels modulo $O(e^{- \delta N})$} We now begin the local analysis of the Bergman-Szeg\H o kernel
above and below, following the notation and terminology of \cite{BBSj}.

Let $B$ be the unit ball in $\C^m$,  and let
$\chi \in C_c^{\infty}(B)$ be a smooth cutoff function  equal to one on the ball of  radius $1/2$. 
Let $M$ be  a \kahler manifold and let $z$ be a fixed point of $M$. Without loss of generality, we assume that the injectivity radius at $z$ is at least $2$.
We identify
$B$ with the unit geodesic  ball around $z$ in $M$ and let $\phi$ be a local \kahler potential for $h$ relative to a local
frame $e_L$ of $L$. Writing a section $s \in H^0(B, L^N)$ in the form $u_N = u e^N_L$ we identify sections with local
holomorphic functions. We define the local $L^2$ norm of the section by
$$\|u_N\|_{N \phi}^2 =\frac{1}{m!} \int_B |u|^2 e^{- N \phi} \omega^m. $$

Let $s$ be a function or a section of a line bundle. We write 
$s=O(R)$, if there is a constant $C$ such that the norm of $s$ is bounded by $CR$.
A family
$K_N(x, y)$ of smooth kernels is a reproducing kernel modulo $O(e^{- \delta N})$ for some $\delta>0$, if there exists an $\epsilon>0$ such that for any fixed
$z$, and any local holomorphic function $u$ on the unit ball $B$, we have
$$
u_N(x) =\int_B(K_N(x,y),\chi(y)u_N(y))_{h^N}dV_M(y)+ O(e^{-\delta N} )\| u_N\|_{N \phi}$$
uniformly in $x\in B_\epsilon=\{x\mid d(x,z)<\epsilon\}$.  Each function $K_N(x,y)$ is called a Bergman kernel modulo $O(e^{- \delta N})$ if it is additionally
holomorphic in $x$.

\subsection{$\tilde \Pi_N(x,y)$ is a Bergman kernel modulo $O(e^{- \delta  N})$}
In this subsection, we prove the following lemma

\begin{lem}\label{lem4}
There exists a constant $\delta>0$ such that $\tilde\Pi_N(x,y)$ is a Bergman kernel modulo $O(e^{- \delta  N})$.
\end{lem}

\begin{proof} 
Let $P_N(x,y)$ be the local reproducing kernel constructed in~\cite{BBSj} and let $P_N$ be the corresponding operator. For any holomorphic    function $u$ on the unit disk $B$, write
\[
u_N-P_N(\chi u_N)=B_N u_N
\]
for operators  $B_N$.
Then by definition
\[
B_Nu_N(x)=O(e^{-\delta N})\|u_N\|_{N\phi}
\]
for some constant $\delta>0$ and for $x\in B_\epsilon$.

Let $\chi_1$ be a cut-off function of the unit ball $B$ such that $\chi_1$ is $1$ in a fixed neighborhood of the origin. We assume that $\textup{supp}\,(\chi_1)\subset B_\epsilon$. Consider the identity
\[
\chi_1(1-\tilde \Pi_N)(\chi u_N)=\chi_1(1-\tilde\Pi_N)(\chi P_N(\chi u_N))+\chi_1(1-\tilde\Pi_N)(\chi (B_N u_N)).
\]
Since $\Pi_N$ is a projection operator, by the definition of $P_N$, we have
\begin{equation}\label{5}
\|\chi_1(1-\tilde\Pi_N)(\chi (B_N u_N))\|_{N\phi}=O(e^{-\delta N})\|u_N\|_{N\phi}.
\end{equation}
According to~\cite{BBSj},
 $P_N(x,y)$ is holomorphic with respect to $x$ when $d(x,y)$ is small. Thus for $y\in B$, we have
\[
\|\chi_1(1-\tilde\Pi_N)(\chi P_N(\cdot,y))\|_{N\phi}=O(e^{-\delta N}).
\]
The proof is the same as that of~\cite[Theorem 3.1]{BBSj} which we  include here for the sake of completeness.

Let $P_{N,y}(x)=P_N(x,y)$. By the construction of the local reproducing kernel (cf. ~\cite[(2.3), (2.7)]{BBSj}), we have
\[
\bar\pa(\chi P_{N,y})= O(e^{-\delta N})
\]
for some $\delta>0$.
By H\"ormander's  $L^2$-estimate (cf. Theorem~\ref{demai}), there exists $v_{N,y}$ such that 
\[
\bar\pa v_{N,y}=\bar\pa(\chi P_{N,y})
\]
with the estimate
\[
\int_{\tilde M}|v_{N,y}|_{\tilde h^N}^2dV_{\tilde M}\leq C\int_{\tilde M}|\bar\pa(\chi P_{N,y})|_{\tilde h^N}^2dV_{\tilde M}=O(e^{-2\delta N}).
\]
By definition, 
\[
(1-\tilde\Pi_N)(\chi P_N(\cdot,y))=(1-\tilde\Pi_N) v_{N,y}.
\]
Therefore for fixed $y$,
\[
\|\chi_1(1-\tilde\Pi_N)(\chi P_N(\cdot,y))\|_{N\phi,y}\leq C\|v_{N,y}\|_{N\phi,y}\leq C e^{-\delta N},
\]
where the norm $\|\cdot\|_{N\phi,y}$ is the $\|\cdot\|_{N\phi}$ norm for the $x$ variable and the pointwise norm for the fixed point $y$.
Thus we obtain
\begin{equation}\label{6}
\|\chi_1(1-\tilde\Pi_N)(\chi P_N(\chi u_N))\|_{N\phi}\leq Ce^{-\delta N}\|u_{N}\|_{N\phi}.
\end{equation}
Combining~\eqref{5} and ~\eqref{6}, we conclude that 
\[
\|\chi_1(1-\tilde\Pi_N)(\chi u_N)\|_{N\phi}=O(e^{-\delta N})\|u_N\|_{N\phi}.
\]

Note that in a neighborhood of the origin, $(1-\Pi_N)(\chi u_N)$ is holomorphic. By working on the ball of radius $N^{-1}$ and the mean-value inequality, the above $L^2$ bound implies the following $L^\infty$ bound
\[
(1-\tilde\Pi_N)(\chi u_N)(x)=O(N^{2m}e^{-\delta N})\|u_N\|_{N\phi}
\]
where $x\in B_\epsilon$. The $N^{2m}$ term can be absorbed by the exponential term if  we further shrink $\delta>0$, and the lemma is proved.

\end{proof}

\subsection{Completing the proof of Theorem~\ref{UPSZEGON}}
Let  $y$ be a fixed point of $\tilde M$.
If we use the same local trivialization of $\tilde L$ at each $\gamma\cdot y$ (that is, these local trivializations are identical via  the one at $\pi(y)\in M$) , then the summation of the right side of the following equation is well-defined
\[
{\Pi}^{\Gamma}_N(x, y): =  \sum_{\gamma \in \Gamma} \tilde{\Pi}_N(x, \gamma \cdot
y).
\]
 By Theorem \ref{AGMON1}, the series converges
for sufficiently large $N$.

We first prove that 
\begin{lem} \label{lem5} There is a constant $\delta>0$ such that  $\Pi^{\Gamma}_N(x, y)$ is a   Bergman   kernel mod $ O(e^{- \delta \sqrt N})$ in the sense that
\[
\int_{\tilde M} (\Pi_N^\Gamma(\cdot,y),\chi(y)u_N(y))_{h^N} dV_M(y)=\chi u_N+O(e^{-\delta \sqrt N})\|\chi u_N\|_{N\phi}
\]
uniformly for any local holomorphic function $u$ on $\tilde M$.
   \end{lem}

\begin{proof}

 Holomorphic sections $s \in H^0(M, L^N)$ lift to $\tilde s$ on $\tilde{M}$ as holomorphic sections
of $\tilde{L}^N$, so our integration lifts to the universal cover.  
Let $z$ be a fixed point of $M$. 
As assumed before, the injectivity radius at $z$ is at least $2$. Let $B$ be the unit ball about $z$. 
The pre-image of $B$ in the universal cover $\tilde M$ are disjoint balls. 
By abuse of notation, we identify  $B$ with   one of these balls and the variables $z,x,y$ are used for points on both $M$ and $\tilde M$.  Let $u$ be a function on $B$. Then  $u$  is regarded as local functions on both $M$ and $\tilde M$.  We regard $u_N$ to be the section of $\tilde L^N$ extended by zero outside $B$.  Thus we have
\begin{align*}
&\int_M (\Pi_N^\Gamma(\cdot,y),\chi(y)u_N(y))_{h^N} dV_M(y)=\sum_{\gamma\in\Gamma}\int_{\tilde M} (\tilde\Pi_N(\cdot,\gamma \cdot y), {\chi(y)u_N(y)} )_{\tilde h^N}dV_{\tilde M}(y).
\end{align*}

Define $d(B, \gamma B)$ to be the distance between $B$ and $\gamma B$. Then  the Agmon estimate gives
$$ |\tilde{\Pi}_N(x, \gamma \cdot y) |_{\tilde h^N} \leq e^{- \beta\sqrt N d(B,\gamma B)}$$
for some $\beta=\beta(M,L,h)>0$ and any $x,y\in B$. Therefore, we have
\[
\sum_{\gamma\neq 1}\left|\int_{\tilde M} (\tilde\Pi_N(\cdot,\gamma \cdot y), {\chi(y)u_N(y)} )_{\tilde h^N}dV_{\tilde M}(y)\right|_{\tilde h^N}\leq
C \sum_{\gamma\neq 1}e^{- \beta\sqrt N d(B,\gamma B)}\|\chi u_N\|_{N\phi}.
\]

By compactness, there is a constant $\sigma>0$ such that 
\[
\sigma\, d(x,\gamma y)\leq d(B,\gamma B)
\]
for any $x,y\in B,\gamma\neq 1$.

Let 
\[
\eta=\inf_{\gamma\neq 1}d(B,\gamma B).
\]
Then we have
\begin{align*}
&
\sum_{\gamma\neq 1}e^{- \beta\sqrt N d(B,\gamma B)}\leq e^{-\frac 12\beta\eta\sqrt N}
\sum_{\gamma\neq 1}e^{-\frac 12 \beta\sqrt N d(B,\gamma B)}\\
&\leq
Ce^{-\frac 12\beta\eta\sqrt N}\int_{\tilde M}e^{-\frac 12 \beta\sqrt N d(z,y)}dV_{\tilde M}(y).
\end{align*}
By the Bishop volume comparison theorem, since the Ricci curvature  has a lower bound, the volume growth of $\tilde M$  is at most exponential. Thus the integral is convergent and we have 
\[
\sum_{\gamma\neq 1}e^{- \beta\sqrt N d(B,\gamma B)}\leq Ce^{-\frac 12\beta\eta\sqrt N}.
\]

Combining the above inequality with lemma~\ref{lem4}, we have
\[
\int_M (\Pi_N^\Gamma(\cdot,y),\chi(y)u_N(y))_{h^N} dV_M(y)=\chi u_N+O(e^{-\delta N}+e^{-\frac 12\beta\eta\sqrt N})\|u\|_{N\phi}.
\]
Since $\delta$ can be chosen  arbitrarily small, the conclusion of the lemma follows.
\end{proof}

\begin{proof}[Proof of Theorem~\ref{UPSZEGON}]
Let
$$R_N(z,w) = \Pi^{\Gamma}_N(z,w)-\Pi_N(z,w). $$
Using the same method as that in Lemma~\ref{lem4} (cf. ~\cite{BBSj}), $\Pi_N(z,w)$ is a Bergman kernel of modulo $O(e^{-\delta N})$ for some $\delta>0$.
Combining with Lemma~\ref{lem5}, we obtain 
\[
\int_M (R_N(\cdot,y),\chi(y)u_N(y))_{h^N} dy=O(e^{-\delta \sqrt N})\|\chi u_N\|_{N\phi}.
\]
Substituting 
\[
u_N(y)=\chi(y)e^{-N\phi(x)/2}R_N(y,x)
\]
into the above equation, 
we get
\[
\int_M \chi(y)|R_N(x,y)|^2_{h^N,x} dV_M(y)=O(e^{-\delta \sqrt N})\sqrt{\int_M\chi^2|R_N(x,y)|_{h^N,x}^2 dV_M(y)},
\]
which implies 
\[
\int_M \chi(y)|R_N(x,y)|^2_{h^N} dV_M(y)=O(e^{-2\delta \sqrt N}).
\]
Note that the above is true uniformly for $x\in B_\epsilon$. Combining this with the Agmon estimate, we obtain
\begin{equation}\label{opernorm}
\sqrt{\int_{M\times M} |R_N(x,y)|^2_{h^N} dV_M(y)dV_M(x)}=O(e^{-\delta\sqrt N}).
\end{equation}
The above left hand side bounds  the norm of the operator $R_N$ defined by the integral kernel $R_N(z,y)$. 

Next we prove  that $R_N^2=R_N$, and hence  $R_N=0$. Recall our  convention of identifying points on $M$ with one of their lifts on $\tilde M$. 
Let $x,w\in M$. Then we have
\begin{align*}
&
\int_M(\Pi_N^\Gamma(x,y), \Pi_N^\Gamma(y,w))_{h^N(y)} e^{-\frac 12 N\phi(w)} dV_{M}(y)\\
&=\int_M\sum_{\gamma,\gamma_1\in\Gamma}
(\tilde\Pi_N(x,\gamma\cdot  y), \tilde\Pi_N(y,\gamma_1 \cdot w))_{h^N(y)} e^{-\frac 12 N\phi(w)} dV_{M}(y).
\end{align*}
Since  $\tilde \Pi(x,y)$ is $\Gamma$ invariant (cf. ~\eqref{invszego}) and since $\Gamma$ acts on $\tilde M$ by isometry, we have
\begin{align*}
&
\int_M(\Pi_N^\Gamma(x,y), \Pi_N^\Gamma(y,z))_{h^N(y)} e^{-\frac 12 N\phi(z)} dV_{M}(y)\\
&=\int_M\sum_{\gamma,\gamma_1\in\Gamma}
(\tilde\Pi_N(x,\gamma \cdot y), \tilde\Pi_N(\gamma\cdot  y,\gamma\gamma_1 \cdot w))_{h^N(y)} e^{-\frac 12 N\phi(w)} dV_{M}(y)\\
&=\int_M\sum_{\gamma,\gamma_1\in\Gamma}
(\tilde\Pi_N(x,\gamma \cdot y), \tilde\Pi_N(\gamma \cdot y,\gamma_1 \cdot w))_{h^N(y)} e^{-\frac 12 N\phi(w)} dV_{M}(y)\\
&=\int_{\tilde M}\sum_{\gamma_1\in\Gamma}
(\tilde\Pi_N(x, y), \tilde\Pi_N( y,\gamma_1 \cdot w))_{\tilde h^N(y)} e^{-\frac 12 N\phi(w)} dV_{\tilde M}(y).
\end{align*}

By the Agmon estimate, any $\tilde\Pi_N( y,\gamma_1\cdot  w) e^{-\frac 12 N\phi(w)}$ is an $L^2$ holomorphic section of $\tilde L^N$. Thus we have
\begin{align*}
&
\int_{\tilde M}\sum_{\gamma_1\in\Gamma}
(\tilde\Pi_N(x, y), \tilde\Pi_N( y,\gamma_1 \cdot w))_{h^N(y)} e^{-\frac 12 N\phi(w)} dV_{M}(y)\\
&
=\sum_{\gamma_1} \tilde\Pi_N(x,\gamma_1 \cdot w)e^{-\frac 12 N\phi(w)}=\Pi^\Gamma_N(x,w)e^{-\frac 12 N\phi(w)}.
\end{align*}
Let $\Pi_N^\Gamma$ be the operator corresponding to the kernel $\Pi_N^\Gamma(x,w)$, then the above computation shows that 
\[
(\Pi_N^\Gamma)^2=\Pi_N^\Gamma.
\]
Since $\Pi_N$ is a projection operator, we have
\[
\Pi_N\Pi_N^\Gamma=\Pi_N^\Gamma.
\]
Since both $\Pi_N$ and $\Pi_N^\Gamma$ are self-adjoint, the above also implies 
\[
\Pi_N^\Gamma\Pi_N=\Pi_N^\Gamma.
\]
As a result, we 
have  $R_N^2=R_N$. Thus  $R_N$ is a projection operator. The operator norm of $R_N$ is $1$ unless $R_N=0$. But by ~\eqref{opernorm}, the norm is less than one for sufficiently large $N$. Thus $R_N=0$ and the theorem is proved. 

\end{proof}

\section{\label{HC} Holomorphic convexity: Proof of Theorem \ref{N}}

We follow the notation of Napier \cite{N} to prove Therorem~\ref{N}.

\begin{proof} 
By the proof of Lemma~\ref{lem5}, the following result  is valid: let $\beta_1>0$ be a fixed positive number, then for $N$ sufficiently large,
\[
\sum_{\gamma\in\Gamma}e^{-\beta_1\sqrt Nd(x,\gamma \cdot y)}\leq C<\infty
\]
for some constant $C$ only depends on the distance $d(x,y)$ of $x,y\in \tilde M$. We shall use the fact below repeatly.

Let $\{y_j\}$ be a divergent sequence of $\tilde M$. By passing a subsequence if needed, we may assume that $\pi(y_j)\to x_0\in M$. 
 By passing a subsequence if needed again, we define the sequence $\{x_j\}$ inductively by the following conditions
\begin{enumerate}
\item for each $x_j$, there exists a $\gamma\in\Gamma$ such that $x_j=\gamma(x_0)$;
\item $d(x_j)\geq j$ for all $j\geq 1$;
\item $\inf_{i<j} d(x_i,x_j)\geq \frac 12 \sup_{i<j} d(x_i,x_j)$ for all $j\geq 1$;
\item $d(x_j,y_j)\to 0$, as $j\to\infty$.
\end{enumerate}

Define
\[
s(x)=\sum_{j=1}^\infty e^{d(x_j)} \tilde \Pi_N(x,x_j).
\]
 Here, as before, we fix a local trivialization of $L$ at $x$ so that $\tilde \Pi_N(x,x_j)$ can be identified as a section of $\tilde L^N$ for each $j$.  We claim that the above series is uniformly convergent on compact sets and hence defines a holomorphic section of $\tilde L^N$. To see this, we use the Agmon estimate to obtain
\[
|s(x)|_{\tilde h^N}\leq C\sum_{j=1}^\infty e^{d(x_j)-\beta\sqrt N d(x,x_j)}.
\]
On any compact set, the norm  can be estimate by
\[
|s(x)|_{\tilde h^N}\leq C\sum_{j=1}^\infty e^{-(\beta\sqrt N-1) d(x_j)}\leq C\sum_{j=1}^\infty e^{-(\beta\sqrt N-1) j}<\infty
\]
for a possibly  larger constant $C$. Thus section $s$ is well defined. 

We verify that $s(x_k)\to\infty$. In fact, using the Agmon's estimate and our construction of the sequence $\{x_j\}$, for any fixed $k$, we have
\[
\left|\sum_{\ell\neq k}e^{d(x_\ell)}\tilde\Pi_N(y_k,x_\ell)\right|_{\tilde h^N}\leq\sum_{\ell\neq k}e^{-(\frac 12\beta\sqrt N-1)d(x_\ell)}<C<\infty
\]
for a constant $C$ independent to $k$. On the other hand, by~\cite{SZ} and the Agmon estimate again, we know that 
\[
\left|\tilde \Pi_N(y_k,x_k)\right|_{\tilde h^N}\geq c e^{-d(x_k,y_k)}
\]
for some constant $c>0$. Thus we have
\[
e^{d(x_k)}\tilde\Pi_N(y_k,x_k)\to\infty
\]
and this completes the proof.

\end{proof}

We remark that such a section $s$ can never be in $ H^0_{L^2}(\tilde{M},
\tilde{L}^N)$.  Indeed,  we note that
$$s(z) = \int_{\tilde{M}} (\tilde{\Pi}_N(z, w), s(w))_{\tilde h^N} dV_{\tilde M}(w) $$
so that if $s$ were square integrable, then
$$|s(z)|_{\tilde h^N}^2 \leq  \int_{\tilde{M}} |\tilde{\Pi}_N(z, w)|^2  dV(w)  \cdot \|s\|_{L^2}^2. $$
We further note that
$$ \int_{\tilde{M}} |\tilde{\Pi}_N(z, w)|^2  dV_{\tilde M}(w) = \tilde{\Pi}_N(z,z).
$$
But  $\tilde{\Pi}_N(z,z)$ is
$\Gamma$ invariant and hence bounded. So square integrable holomorphic  sections are
automatically bounded and we get a contradiction.

\section{\label{SURJ} Application to surjectivity of  Poincar\'e series: Proof of Theorem~\ref{POINCARE}}
We now give a  simple proof of surjectivity when   Theorem~\ref{UPSZEGON}  and Theorem \ref{AGMON1}  are valid:

\begin{proof}

  We define
the  coherent state (or peak section) $\Phi^w_{N} \in H^0(M, L^N)$ centered at $w$ by 
$$\Phi_{N}^w (z) =\Pi_{N}(z,w). $$
 By Theorem \ref{UPSZEGON}, we have
\begin{equation} \Phi_{N}^w(z) = \sum_{\gamma \in \Gamma} \tilde{\Phi}^{\tilde{w}}_{N}( \gamma\cdot  z) 
= P  \tilde{\Phi}^{\tilde{w}}_{N} (z),\end{equation}
where 
$$\tilde \Phi_{N}^w (z) =\tilde\Pi_{N}(z,w). $$

For any $s\in H^0(M,L^N)$,
\[
\langle s, \Phi_{h^N}^w (z)  \rangle_{h^N} = s(z). 
\]
Therefore, 
\[
s(z)=P(\langle s,\tilde{\Phi}^{\tilde{w}}_{N} (z)\rangle)_{h^N}
\]
is written as the Poincar\'e series and the theorem is proved.
\end{proof}

\bibliographystyle{abbrv}

\begin{thebibliography}{BBB}

\bibitem[Ag]{Ag}
S.  Agmon,
\newblock {\em Lectures on exponential decay of solutions of second-order
  elliptic equations: bounds on eigenfunctions of {$N$}-body {S}chr\"odinger
  operators}, volume~29 of {\em Mathematical Notes}.
\newblock Princeton University Press, Princeton, NJ, 1982.


\bibitem[A]{A}
M.~F. Atiyah,
\newblock Elliptic operators, discrete groups and von {N}eumann algebras.
\newblock In {\em Colloque ``{A}nalyse et {T}opologie'' en l'{H}onneur de
  {H}enri {C}artan ({O}rsay, 1974)}, pages 43--72. Ast\'erisque, No. 32--33.
  Soc. Math. France, Paris, 1976.


\bibitem[BBSj]{BBSj}
R. Berman and  B.  Berndtsson and J. Sj{\"o}strand,
\newblock A direct approach to {B}ergman kernel asymptotics for positive line
  bundles.
\newblock {\em Ark. Mat.}, 46(2):197--217, 2008.

\bibitem[Bo]{Bo}
B. Berndtsson,
\newblock An eigenvalue estimate for the {$\overline\partial$}-{L}aplacian.
\newblock {\em J. Differential Geom.}, 60(2):295--313, 2002.


\bibitem[BouSj]{BouSj}
L.~Boutet~de Monvel and J.~Sj{\"o}strand,
\newblock Sur la singularit\'e des noyaux de {B}ergman et de
  {S}zeg\H{o}.
\newblock In {\em Journ\'ees: \'{E}quations aux {D}\'eriv\'ees {P}artielles de
  {R}ennes (1975)}, pages 123--164. Ast\'erisque, No. 34--35. Soc. Math.
  France, Paris, 1976.



\bibitem[Ca]{Ca}  F. Campana,
\newblock Remarques sur le rev\^etement universel des vari\'et\'es
  k\"ahl\'eriennes compactes.
\newblock {\em Bull. Soc. Math. France}, 122(2):255--284, 1994.

\bibitem[CD]{CD}  F.  Campana and J. P.  Demailly,
\newblock Cohomologie {$L^2$} sur les rev\^etements d'une vari\'et\'e complexe
  compacte.
\newblock {\em Ark. Mat.}, 39(2):263--282, 2001.







\bibitem[Ch]{Ch}
M. Christ,
\newblock On the {$\overline\partial$} equation in weighted {$L^2$} norms in
  {${\bf C}^1$}.
\newblock {\em J. Geom. Anal.}, 1(3):193--230, 1991.


\bibitem[Del]{Del} H.  Delin,
\newblock Pointwise estimates for the weighted {B}ergman projection kernel in
  {$\mathbf C^n$}, using a weighted {$L^2$} estimate for the
  {$\overline\partial$} equation.
\newblock {\em Ann. Inst. Fourier (Grenoble)}, 48(4):967--997, 1998.


\bibitem[D1]{D1}  
J. P.  Demailly,
\newblock Estimations {$L^{2}$} pour l'op\'erateur {$\bar \partial $} d'un
  fibr\'e vectoriel holomorphe semi-positif au-dessus d'une vari\'et\'e
  k\"ahl\'erienne compl\`ete.
\newblock {\em Ann. Sci. \'Ecole Norm. Sup. (4)}, 15(3):457--511, 1982.


\bibitem[D2]{D2} J. P. Demailly, 
$ L^2 $ estimates for the $\dbar$ operator on complex manifolds, Notes de cours, Ecole d'\'et\'e Math\'ematiques (Analyse Complexe), Institut Fourier, Grenoble, Juin 1996 (online at http://www-fourier.ujf-grenoble.fr/~demailly/books.html).


\bibitem[DM]{DM}
Tien-Cuong Dinh, George Marinescu, and Viktoria Schmidt.
\newblock Equidistribution of zeros of holomorphic sections in the non-compact
  setting.
\newblock {\em J. Stat. Phys.}, 148(1):113--136, 2012.

\bibitem[Donn]{donn}
Harold Donnelly.
\newblock Elliptic operators and covers of {R}iemannian manifolds.
\newblock {\em Math. Z.}, 223(2):303--308, 1996.


 \bibitem[Ear]{Ear}  
 C.  Earle,
\newblock Some remarks on {P}oincar\'e series.
\newblock {\em Compositio Math.}, 21:167--176, 1969.



\bibitem[EC]{ec}
D. Ebin, J. Cheeger
\newblock {C}omparison theorems in {R}iemannian geometry.
\newblock {\em North-Holland Mathematical Library, Vol. 9}, 1975.

\bibitem[E]{E} P. Eyssidieux,
\newblock Invariants de von {N}eumann des faisceaux analytiques coh\'erents.
\newblock {\em Math. Ann.}, 317(3):527--566, 2000.


\bibitem[EKPR]{EKPR}
P.~Eyssidieux, L.~Katzarkov, T.~Pantev, and M.~Ramachandran.
\newblock Linear {S}hafarevich conjecture.
\newblock {\em Ann. of Math. (2)}, 176(3):1545--1581, 2012.

\bibitem[G1]{G1} M. Gromov, Sur le groupe fondamental d'une vari\'et\'es K\"ahlerienne, C.R. Acad. Sci. Paris 308 S\'erie 1: 67--70, 1989.

\bibitem[G2]{G2} M. Gromov,
\newblock K\"ahler hyperbolicity and {$L_2$}-{H}odge theory.
\newblock {\em J. Differential Geom.}, 33(1):263--292, 1991.


\bibitem[GHS]{GHS} M.~Gromov and  G.~Henkin  and M.~Shubin, 
\newblock Holomorphic {$L^2$} functions on coverings of pseudoconvex manifolds.
\newblock {\em Geom. Funct. Anal.}, 8(3):552--585, 1998.





\bibitem[Kai]{Kai} V.~A. Ka{\u\i}manovich,
\newblock Harmonic and holomorphic functions on coverings of complex manifolds.
\newblock {\em Mat. Zametki}, 46(5):94--96, 1989.



\bibitem[K]{K} J.  Koll{\'a}r, 
\newblock {\em Shafarevich maps and automorphic forms}.
\newblock M. B. Porter Lectures. Princeton University Press, Princeton, NJ,
  1995.


\bibitem[L]{L} N.  Lindholm,
\newblock Sampling in weighted {$L^p$} spaces of entire functions in {${\mathbb
  C}^n$} and estimates of the {B}ergman kernel.
\newblock {\em J. Funct. Anal.}, 182(2):390--426, 2001.




\bibitem[MaMa]{MaMa} X. Ma and G.  Marinescu,
\newblock {\em Holomorphic {M}orse inequalities and {B}ergman kernels}, volume
  254 of {\em Progress in Mathematics}.
\newblock Birkh\"auser Verlag, Basel, 2007.


\bibitem[M]{M}  A. Manning,
\newblock Relating exponential growth in a manifold and its fundamental group.
\newblock {\em Proc. Amer. Math. Soc.}, 133(4):995--997 (electronic), 2005.

\bibitem[Mi]{MI}  J. Milnor,  A note on curvature and fundamental group. J. Differential Geometry, 2: 1-7,1968.


\bibitem[N]{N} T. Napier,
Convexity properties of coverings of smooth projective varieties.
Math. Ann., 286, no. 1-3:  433--479, 1990.

\bibitem[NR]{NR} T. Napier and M. Ramachandran,
\newblock The {$L^2\ \overline \partial$}-method, weak {L}efschetz theorems,
  and the topology of {K}\"ahler manifolds.
\newblock {\em J. Amer. Math. Soc.}, 11(2):375--396, 1998.

\bibitem[SZ]{SZ} B. Shiffman and S. Zelditch,
\newblock Asymptotics of almost holomorphic sections of ample line bundles
   on symplectic manifolds.
   \newblock {\em J. Reine Angew. Math.}, 544: 181-222, 2002.

\bibitem[Y1]{Y1}
S-K. Yeung.
\newblock Very ampleness of line bundles and canonical embedding of coverings
  of manifolds.
\newblock {\em Compositio Math.}, 123(2):209--223, 2000.


\bibitem[Y2]{Y2} S-K. Yeung,
\newblock Effective estimates on the very ampleness of the canonical line
  bundle of locally {H}ermitian symmetric spaces.
\newblock {\em Trans. Amer. Math. Soc.}, 353(4):1387--1401 (electronic), 2001.

\bibitem[Y3]{Y3} S-K. Yeung,
\newblock Betti numbers on a tower of coverings.
\newblock {\em Duke Math J.}, 73(1): 201--226, 1994.


  \bibitem[Z2]{Z2} 
S. Zelditch,
\newblock Szeg{\H o} kernels and a theorem of {T}ian.
\newblock {\em Internat. Math. Res. Notices}, (6):317--331, 1998.
\end{thebibliography}

\end{document}